\documentclass[12pt]{amsart}
\usepackage{amscd}
\usepackage{verbatim}
\usepackage{amssymb, amsmath, amsthm, amscd,ifthen}
\usepackage[dvips]{graphics}
\usepackage[cp866]{inputenc}
\usepackage{graphicx}
\usepackage{epsfig}

\usepackage{amsmath,amssymb,amscd,}
\usepackage[all]{xy}

\usepackage{amssymb}

\newtheorem{thm}{Theorem}[section]

\newtheorem{prop}[thm]{Proposition}

\newtheorem{lemma}[thm]{Lemma}

\theoremstyle{definition}

\newtheorem{dfn}[thm]{Definition}

\newtheorem{stat}{Statement}

\title
{Extremal configurations of robot arms in three dimensions}

\author{ Dirk Siersma}

\address{ University of Utrecht, Mathematisch Instiuut;  e-mail: D.Siersma@uu.nl
}


\begin{document}

\maketitle \setcounter{section}{0}

\begin{abstract}We define a volume function for a robot arm in $\mathbb R^3$ and give geometric conditions for its critical points.
\end{abstract}

\section{Introduction}

Linkages are flexible 1-dimensional structures, where edges are straight intervals of a fixed length, where flexes are allowed at vertices. For general properties of linkages we refer to
\cite{conndem},\cite{milskap} and \cite{milskap1}.

Recently G. Khimshiashvili, G. Panina, their co-workers
 and the author investigated various extremal problems on the moduli spaces of
linkages.
An important part of that studies concerned the cyclic
configurations of planar polygonal linkages and open robot arms considered as the
critical points of the oriented area function \cite{khi4}, \cite{khipan} , \cite{kpsz2}, 
 \cite{khisie} and \cite{panzh}.

The aim of the current paper is to generalize these statements to the
3-dimensional case. We will give a  geometric
description of the critical configurations in the case of oriented
volume in 3D. The extremal arms consist of planar circular contributions combined with zigzags (theorem  \ref{thm:crit-arms}).
For computational reasons we consider the signed volume function on a parameter space and not on the moduli space. The isotropy groups  of oriented isometries acting on this parameter space are not constant. We study this effect for the 3-arm and show in that case:

The  oriented moduli space of 3-arms in  ${\mathbb R}^3$ is a
3-sphere. The Volume function is an exact topological Morse function on this space
with precisely two Morse critical points.

This research was supported through the
programme "Research in Pairs"  by the Mathematisches Oberwolfach in 2010. It's our special pleasure to
acknowledge the excellent working conditions and warm hospitality of
the whole staff of the institute during our visit in November 2010. 
The outcome of the project was published in a Oberwolfach preprint \cite{kpsz1}. Sections 6-9 are the source of the current paper. Later G. Panina \cite{pan3} and \cite{khrpan} obtained results for the volume function on closed polygons, including information about Morse indices.

I thank
G. Khimshiashvili, G. Panina and A. Zhukova for useful discussions their contributions to this paper.

\section{Preliminaries and notation}\label{section_preliminaries}

An \textit{ $n$-linkage} is a sequence of positive numbers
$l_1,\dots ,l_n$. It should be interpreted as a collection of rigid
bars of lengths $l_i$ joined consecutively by revolving joints in a
chain, either open or closed. Open linkages are sometimes called
\textit{ robot arms}. We study the flexes of the both types of chain
with allowed self-intersections. This is formalized in the following
definitions.
\begin{dfn}

For an open  linkage $L$, \textit{a configuration} in the
Euclidean space $ \mathbb{R}^d$ is a sequence of points
$R=(p_1,\dots,p_{n+1}), \ p_i \in \mathbb{R}^d$ with
$l_i=|p_i,p_{i+1}|$ modulo the action of orientation preserving
isometries.
We also call $R$ \textit{an open chain}.

 The set $M_d^\circ (L)$ of all such
configurations is \textit{the moduli space, or the configuration
space of the robot arm }$L$.

 For a closed polygonal linkage, we claim in addition that the last
point coincides with the first point: a configuration of the linkage
$L$ in the Euclidean space $ \mathbb{R}^d$ is a sequence of points
$P=(p_1,\dots,p_n), \ p_i \in \mathbb{R}^d$  with
$l_i=|p_i,p_{i+1}|$ for $i=1,..,n-1$ and $l_n=|p_n,p_{1}|$. As
above, the action  of orientation preserving isometries is factored
out.
We also call $P$ \textit{a closed chain} or a \textit{polygon}.

The set $M_d(L)$ of all such configurations  is \textit{the moduli
space, or the configuration space of the polygonal linkage }$L$.

\end{dfn}

In \cite{khipan} and
\cite{khisie}  the 2-dimensional case was treated with the signed area function
on the configuration space. We recall some definitions and results.

\begin{dfn} \label{Dfn_area}

    The \textit{signed area} of a polygon $P$ with the vertices \newline $p_i = (x_i,
y_i)$  is defined by
$$2A(P) = (x_1y_2 - x_2y_1) + \ldots + (x_ny_1 - x_1y_n).$$

		The \textit{signed area} of an open chain  with the vertices $p_i = (x_i,
y_i)$  is defined by
$$2A(P) = (x_1y_2 - x_2y_1) + \ldots + (x_ny_{n+1} - x_{n+1}y_n)+(x_{n+1}y_1 - x_1y_{n+1}).$$
In other words, we add one more edge that turns an open chain to a
closed polygon and take the signed area of the polygon.

\end{dfn}

\begin{dfn}

     A polygon  $P$  is called \textit{cyclic} if all its vertices $p_i$
lie on a circle.

		A robot arm  $R$ is  called \textit{diacyclic} if all its vertices
$p_i$ lie on a circle, and $p_1p_{n+1}$ is the diameter of the
circle.

\end{dfn}

 Cyclic polygons and cyclic open chains arise
 as critical points of the signed area:

\begin{thm}\label{Thm_critical_are_cyclic} (\cite{khipan},
\cite{khisie})

    Generically, a polygon $P$ is a critical point of the
signed area function $A$  iff $P$ is a cyclic configuration.

    Generically, an  open robot arm $R$ is a critical point of the
signed area function $A$  iff $R$ is a diacyclic configuration.
      \qed
\end{thm}

\section{About $3$-arm in ${\mathbb R^3}$} \label{3arm}
Before we treat in the next section open linkages with $n$ arms in
${\mathbb R^3}$,
we study here $3$-arms in $\mathbb{R}^3$.\\
Let us  fix some notation. The arm vectors are:
 $ a = (1,0,0)$, $b$ and $c$ of length $|a|,|b|,|c|$. \\
A spatial arm is constructed as follows: we take the segments from
$O$ to the end points $A$, $B$, $C$ of $a$, $a+b$, $a+b+c$. This
yields a tetrahedron $OABC$.

\begin{dfn}
We define the \textit{signed volume} $V$ of the $3$-arm as the
triple vector product:
          $$ V = [a,a+b,a+b+c] = [a,b,c]. $$
          \end{dfn}

We intend to study $V$ on several parameter spaces:
\begin{itemize}
\item  On $S^2 \times S^2$,
\item  On $S^1 \times S^2$,  where we fix the vector $b$ to lie in the $xy$ plane,
\item  On the moduli space $M^o_3$  (mod the $SO(3)$ action).
\end{itemize}
In each of these cases  critical points may be different. We intend
to compare the critical points and the Morse theory for the three
cases.

\subsection{On $S^2 \times S^2$}
Before starting we define some special positions of the $3$-arm:
\begin{itemize}
\item  \textit{Tri-orthogonal}: The vectors $a$, $b$, $c$ are tri-orthogonal;
equivalently: the sphere with diameter $OC$ contains also the points $A$ and $B$,
\item \textit{Degenerate}: The arm lies in a two-dimensional subspace,
\item \textit{Aligned}: The arm is contained in a line.
\end{itemize}

\begin{prop}

The signed area $V: S^2 \times S^2 \rightarrow  {\mathbb R} $ has
the following critical points:
\begin{itemize}
\item    Tri-orthogonal arms (maximum, resp minimum).
 These are  Bott-Morse critical points with transversal index $3$ and critical value $\pm |a||b||c|$.
\item Isolated points, corresponding to the aligned configurations. Here $V$ has Morse index $2$ and the critical value $0$.
\end{itemize}

\end{prop}

\begin{proof}

We use coordinate systems on the spheres; we take partial
derivatives with respect to all coordinates. We denote the partial
derivatives of $b$ by $\delta_1b$ and $\delta_2b$. Both are non-zero
and orthogonal to $b$. We take  partial derivatives of $V = [a,b,c]
$ in the $(\delta_1b,\delta_2b)$ directions: $[a,\delta_1b,c] = 0$
and   $[a,\delta_2b,c] = 0$.

We will shorten this to $[a,\dot{b},c] = 0$  meaning that the
equation holds for all vectors in the tangent space of $b$ (which is
orthogonal to $b$ and spanned by $\delta_1b$ and $\delta_2b$). In
this way we get:
$$ [a,\dot{b},c] = 0 , \ \  [a,b,\dot{c}] = 0. $$
For both equations we will consider two cases:

\bigskip

\begin{tabular}{|l|l|l|}
\hline
 equation            &   ortho condition   &  parallel condition\\
 \hline
                      & $a \times c \ne o$  &  $a \times c = o$ \\
 $[a,\dot{b},c] = 0$  &   equivalent to    & equivalent to \\
                      &  $b \perp a$ and $b \perp c$ & $a \parallel c$ \\
  \hline
                      &     $a \times b \ne o$ &    $a \times b =  o$   \\
 $  [a,b,\dot{c}] = 0$  &      equivalent to    & equivalent to \\
                       & $c \perp a$ and $c \perp b$ & $a \parallel b$ \\
                       \hline
\end{tabular}

\bigskip

The combination of the two ortho conditions gives the tri-orthogonal
case of the proposition; combining the two parallel conditions is
the aligned case. Combining one ortho condition with the other
parallel condition gives a contradiction.
\end{proof}
Next we describe the type of the critical points. For the positively
oriented tri-orthogonal case we get a maximum. Due to the remaining
$SO$-action the singular set is an $S^1$, and its transversal Morse
index is $3$. The other orientation gives a minimum on $ S^1$ with
the transversal Morse index $0$. The aligned configurations ($4$
cases) occur in isolated points. In all these cases we have index
$2$. We check  the Bott-Morse formula:
$$ \sum t^{\lambda(C)} P(C) - P(M) = (1+t)R(t)   $$
where $R(t)$ must have non-negative coefficients. In our case we
have
$$t^3(1+t)+(1+t)+(1+t) + 4t^2 - (t^4+2t^2+1) = t^3+2t^2+t =(1+t)(t^2+t),$$
so this is OK.\qed

\subsection{On $S^1 \times S^2$}

After a rotation we can always assume that  $b$ lies in the
$xy$-plane. We consider $SO$-action, that fixes this plane.

\begin{prop}

The signed volume $V: S^1 \times S^2 \rightarrow  {\mathbb R} $ has
the following critical points:
\begin{itemize}
\item  $4$ points, corresponding to  tri-orthogonal arms ($2$ maxima, respectively $2$ minima).\\ At these points $V$ has critical value $0$.
\item Two circles corresponding to degenerate  configurations. where $a$ and $b$ are aligned and
$c$ is free to move in the $xy$-plane. At these points $V$ has
Bott-Morse critical points with transversal index $1$.
\end{itemize}

\end{prop}

The proof is a straight forward computation \cite{kpsz1}.

We check the result with Bott-Morse formula:\\
$ 2t^3 + 2 + 2t(1+t) - (t^3 + t^2 + t + 1) = t^3 + t^2 + t + 1 = (t+1)(t^2+1)$ .\\

Note the difference between the situation on $S^2 \times S^2$ and on
$S^1 \times S^2$.

\subsection{On the moduli space $M^o_3$ .}

 This moduli space is homeomorphic to $S^3$. This is shown in \cite{merm}. We return to this later in this paper.
 An outline is as follows:
 First construct the non oriented moduli space and show that this is a topological 3-ball.
The sphere $S^3$ appears as a gluing  of two such balls along their
common boundary. This boundary consists of degenerate arms (those
who are not the maximal dimension).

The function $V$ will be studied separately on the two hemispheres,
each of whom has exactly one  Morse point. Near the common boundary
one can show that $V$ glues to a topologically regular function. In
Section \ref{SectionGram}  we give details and prove the following:

\begin{thm}
The  oriented moduli space of 3-arms in  ${\mathbb R}^3$ is a
3-sphere. $V$ is an exact topological Morse function on this space
with precisely two Morse critical points.\qed
\end{thm}

Note that the critical points with $V=0$, which we got before in the
cases with parametrization $S^2\times S^2$ or $S^1 \times S^2$,
are no longer (topological) critical on the moduli space.

\section{About n-arms in ${\mathbb R^3}$}

There is no unique way to attach a  volume to a
polygonal chain. We take one special situation as starting point
for our definition of (signed) volume in case of a n-arm in $\mathbb R^3$.
The following picture where all simplices contain $a=b_1$
illustrates this definition. 

The relation with the volume of the convex hull can be lost, especially when the combinatorics of the convex hull changes.

\vspace{0cm}
   
\begin{center}
\includegraphics[ width=20em]{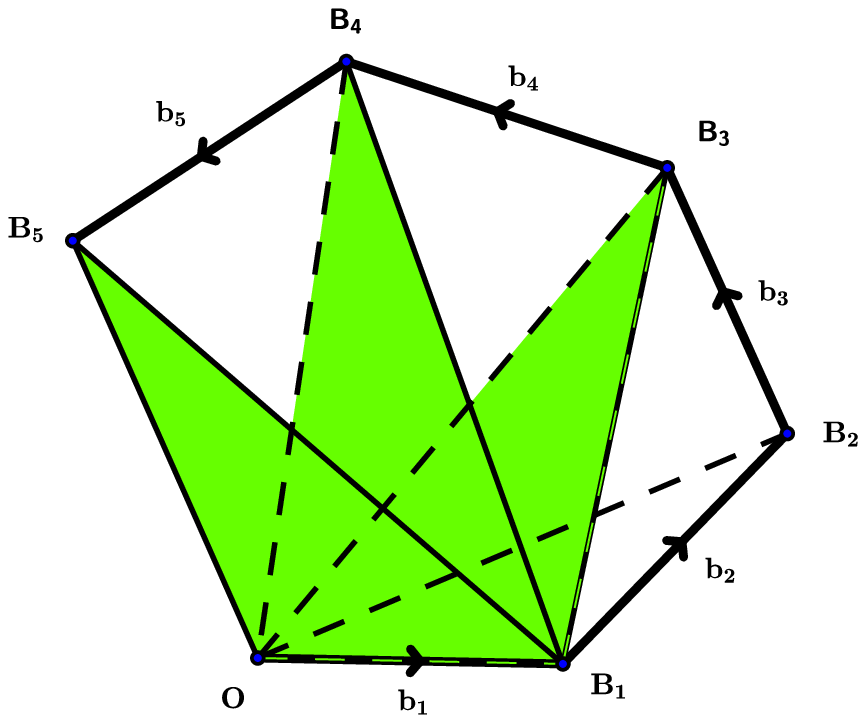}
\end{center}

\vspace{0cm}

\begin{dfn}Let an n-arm be given by the vectors $ b_1,\cdots, b_n$. The
vertices are $O,B_1,\cdots, B_n$. We fix $b_1= a$ (as before). We
denote $c_k= \sum_{i=1}^{k}b_i$ (the endpoint of this vector is
$B_k$). The \textit{signed volume function} is defined as
$$V = \sum_{k=1}^{n-1} [ b_1, c_k,c_{k+1}],$$
which can be rewritten as:
$$ V = [b_1,b_2,b_3] +  [b_1, b_2+b_3,b_4] + [b_1, b_2+b_3 + b_4 ,b_5] + \cdots
        [b_1,b_2 + \cdots b_{n-1}, b_n] .$$
        \end{dfn}
N.B. Note that this signed volume is essentially the signed area of
the projection onto the plane orthogonal to $b_1$.

\begin{lemma} {\bf (Mirror lemma)}
Let two arms differ on a permutation of the arms $2,\dots,n$. Then
there  exits a bijection (by 'mirror-symmetry') between their
"moduli spaces" which preserves the signed volume function.
Consequently this bijection preserves critical points and their
local (Morse) types.
\end{lemma}
\begin{proof}
As in the planar case \cite{kpsz2}.
\end{proof}

The conditions for critical points are:
$$ [b_1,\dot{b_2}, b_3] + [b_1,\dot{b_2},b_4] + \cdots + [b_1,\dot{b_2}, b_n]
 = [b_1, \dot{b_2}, b_3 + \cdots + b_n] = 0 .$$
$$  [b_1,b_2,\dot{b_3}] + [b_1,\dot{b_3},b_4] + \cdots [b_1, \dot{b_3},b_n]
   = [b_1, b_2 - (b_4 + \cdots + b_n), \dot{b_3}] = 0 .$$

The $r^{th}$ -derivative gives the following:
$$ [b_1, b_2+ \cdots + b_{r-1}, \dot{b_r}] + [b_1,\dot{b_r},b_{r+1}] + \cdots +  [b_1,\dot{b_r},b_n] = $$
$$ = [b_1, b_2 + \cdots  + b_{r-1} - (b_{r+1} + \cdots + b_n), \dot{b_r}] = 0 .$$

There are two cases for any $2 \le r \le n$ (which we call {\it
ortho} and {\it parallel}):
\begin{itemize}
\item[$\bullet$ case $O_r$:]
$$ b_1 \times (( b_2 + \cdots  + b_{r-1}) - (b_{r+1} + \cdots + b_n)) \ne 0 .$$
Hence  we have the following orthogonalities
$$ b_r \perp b_1 \  \wedge \ \
 b_r \perp (b_2 + \cdots + b_{r-1}) - (b_{r+1}+ \cdots + b_{n}). $$
\item[$\bullet$ case $P_r$:]  $$ b_1 \times (( b_2 + \cdots  + b_{r-1}) - (b_{r+1} + \cdots + b_n)) = 0 ,$$
which means that $(b_2 + \cdots  + b_{r-1}) - (b_{r+1} + \cdots +
b_n) \in  {\mathbb R} b_1.$
\end{itemize}

Next we decompose vectors into  their ${\mathbb R} b_1$-component
and its orthogonal complement:
$$ b_r = b_r' + b_r^{\perp}$$

\begin{lemma}
For all $r= 2, \cdots , n$:
$$ b_r^{\perp} \perp (b_2^{\perp} + \cdots + b_{r-1}^{\perp}) - (b_{r+1}^{\perp}+ \cdots + b_{n}^{\perp}) $$
and also
$$ (b_2^{\perp} +  \cdots + b_{r-1}^{\perp} ) \perp (b_r^{\perp} + \cdots b_n^{\perp} )\ \ \ (*) $$
\end{lemma}

For any critical point of the signed volume function on n-arms in
${\mathbb R}^3$ one can consider the projection of the arm onto the
hyperplane orthogonal to $b_1$.
\begin{prop}
The vertices of this planar $(n-1)$-arm $b_2^{\perp}, \dots,
b_n^{\perp}$ lie on a circle with diameter the interval
$B_1B_n^{\perp}$ from the start point to the end point of this arm.
This configuration corresponds to a critical point of such arms (but
with fixed lengths) under the signed area function.\qed
\end{prop}
Note that in general we don't have  fixed lengths of the projections
and that projections can be "degenerate".

We next treat  several cases of the spatial situations and after
that state the general result in Theorem \ref{thm:crit-arms}.

\begin{subsection}
{\bf Full ortho case: }$O_r$ for all $r=2,\dots,n .$\\
Now $b_r = b_r^{\perp}$. So we have:
\begin{stat}
The  critical points of the signed volume function on n-arms in
${\mathbb R}^3$ are exactly those configurations, where all vertices
(including the first $O$ and the last $B_r$) are on a sphere with
diameter $OB_r$; the first arm is perpendicular to the all other
arms. Delete the first arm: the vertices of this planar $(n-1)$-arm
lie on a circle with $B_1B_r$ as the diameter. This configuration
corresponds precisely to a critical point of such arms under the
signed area function. Moreover,
$$V = |b_1|\cdot sA.$$

\end{stat}
\end{subsection}

\begin{subsection}
{\bf Full parallel case:} $P_r$ for all $r=2,\dots,n.$\\
\newline If $n$ is odd we find
$b_r \in {\mathbb R}b_1$ ($r=2,\dots, n$). \newline If $n$ is even
we find $b_r + b_{r+1} \in {\mathbb R}b_1$ ($r=2,\dots, n-1$).
\begin{stat}
Critical points of $V$ are aligned configurations if $n$ is  odd and
1-parameter families of zigzags if $n$ is even. Zigzags are arms,
which project all to the same interval (see Fig. \ref{FigA}, right).
\end{stat}
Zigzags also contain the aligned configuration. In a zigzag the
lengths of the projections can vary the  from 0 to
the minimum lengths of $b_2,\dots,b_r$.\\
Both full cases (see Fig. \ref{FigA}) have the property that
solutions exists for all length vectors.
\end{subsection}

\begin{figure}
\centering
\includegraphics[width=12 cm]{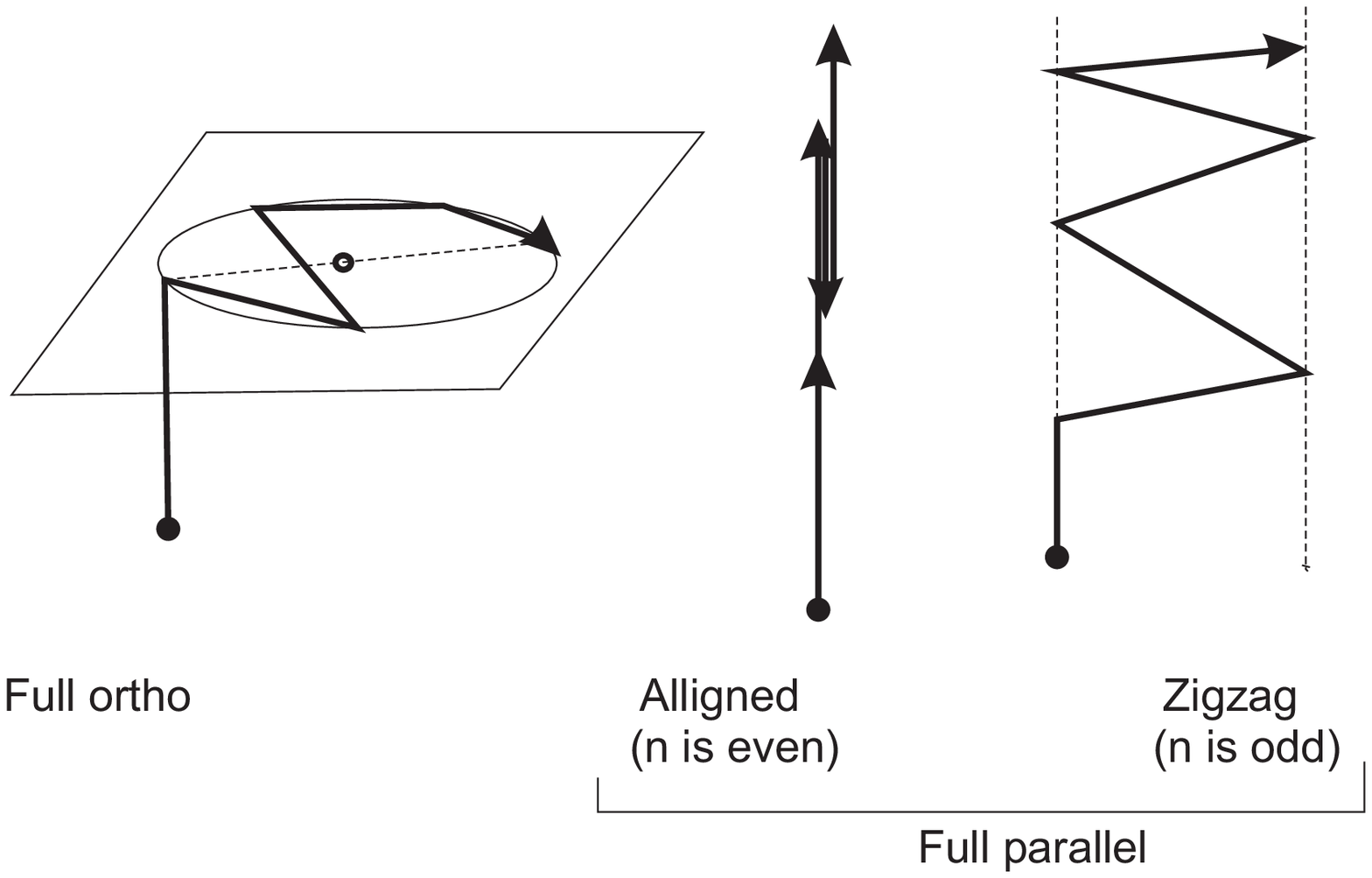}
\caption{ }\label{FigA}
\end{figure}

\subsection{ General case: $n-k$ parallel conditions, and  $k-1$ ortho conditions}
We can assume (due to the mirror lemma) that the last $n-k$
conditions are parallel. That is, we have
$$ b_2 + \cdots + b_{k} + b_{k+1}^{\perp} + \cdots +b_{n-1}^{\perp} = 0 $$
together with
$$ b_{k+1} + b_{k+2} \in {\mathbb R}b_1 ,\cdots , b_{n-1} + b_n \in {\mathbb R}b_1 .$$
So
$$ b_{k+1}^{\perp} + b_{k+2}^{\perp} = 0 ,\cdots , b_{n-1}^{\perp} + b_n^{\perp} = 0. $$
This has the following consequences:

\begin{itemize}
\item The $b_{k+1}^{\perp}, \cdots, b_{n}^{\perp}$ are diameters of the critical circle,
\item If $n-k$ is even, then $ b_2 + \cdots + b_{k} + b_{k+1}^{\perp} = 0 $.\\
The ($k-1$)-arm $b_2,\cdots,b_{k}$ is an open planar diacyclic chain
(\textit{diameter condition}).
\item If $n-k $ is odd, then $ b_2 + \cdots + b_{k} = 0 $.
The ($k-1$)-arm $b_2,\cdots,b_{n-k-1}$ is a closed planar cyclic polygon
(\textit{closing condition}).
\end{itemize}
In both cases (odd and even) the projections of the vertices lie on
a circle (see Fig. \ref{FigB}). There are only finite number of
these circles possible. For a
 realization it is necessary that $|b_i| \ge R$ (radius of circle) if $k+1 \le i \le n$.

\begin{figure}
\centering
\includegraphics[width=12 cm]{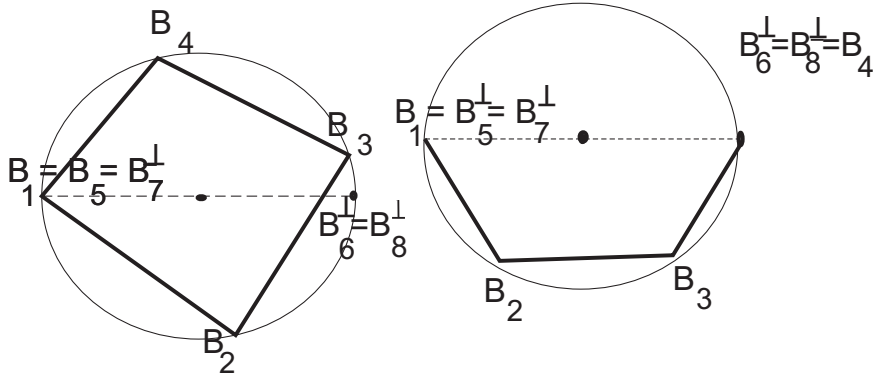}
\caption{Projected vertices are on a circle. }\label{FigB}
\end{figure}


 The above discussion shows the following:
\begin{thm} \label{thm:crit-arms}
The critical points of $\ V$  up to "mirror-symmetry" are as follows (see Fig. \ref{Figmixed}): \\
There exits a division of the n-arm into a sub-arm $b_1$, a subarm
$b_2,\dots,b_k$ and a subarm $b_{k+1},\dots,b_n$ such that:
\begin{itemize}
\item $b_1$ is orthogonal to each of $b_2,\dots,b_k$ (which lie in a plane ${\mathbb
R}b_1^{\perp}$).
\item The vertices of the arm  $b_2,\dots,b_k$ lie on a circle, satisfying
\begin{itemize}
\item the closing condition if $n-k=$ odd,
\item the diameter condition if $n-k $ = even.
\end{itemize}
\item The arm $b_{k+1},\dots,b_n$ is a zigzag, which projects orthogonally to the diameter of the circle.\qed
\end{itemize}

\end{thm}

\begin{figure}
\centering
\includegraphics[width=12 cm]{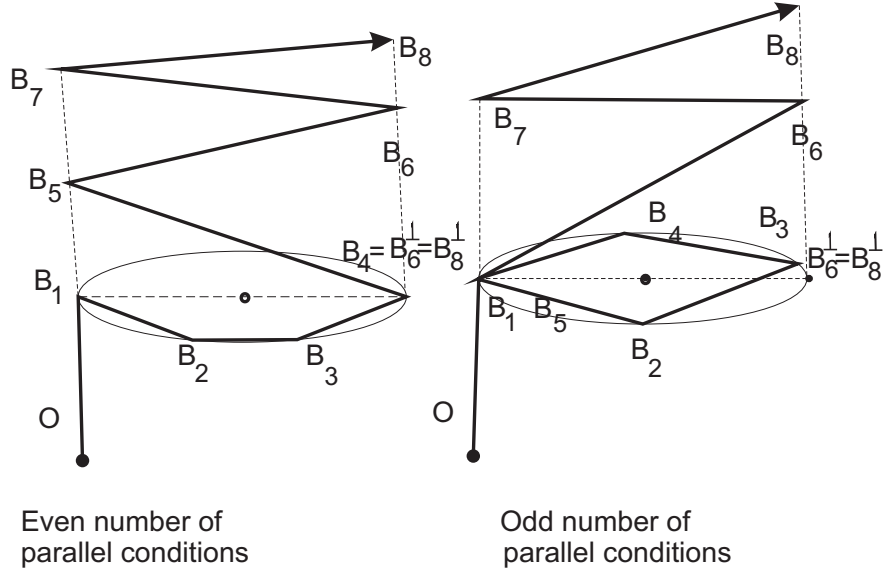}
\caption{ Solutions in the general case.}\label{Figmixed}
\end{figure}

\section{About n-arms in ${\mathbb R^3}$; projection on planes}

 As mentioned befor the  signed volume is essentially the signed area of
the projection onto the plane orthogonal to $b_1$. The same reasoning can be applied to more general projections.
We consider in ${\mathbb R^3}$ a vector $p$, which is the direction of
the orthogonal projection on a plane ${\mathbb R}p^{\perp}.$

Let the n-arm be given by the vectors $ b_1,\cdots, b_n$. The
vertices are $O,B_1,\cdots, B_n$.

Define the signed Projected Area function as follows:
$$ PA = [p,b_1,b_2] +[p,b_1+b_2,b_3] +  [p, b_1+ b_2 + b_3,b_4] + $$
$$[p, b_1+ b_2+ b_3 + b_4 ,b_5] + \cdots + [p, b_1+ \cdots + b_{n-1},
b_n] . $$

{\bf We fix first both the positions of $p$ and $b_1$!}.\\
We assume that $p \times b_1 \ne 0$.\\

\begin{thm} \textbf{(Projection with fixed $p$ and $b_1$)}
The critical points of $PA$  up to "mirror-symmetry" are as follows: \\
There exits a division of the n-arm into two subarms $b_1,\dots,b_k$
and $b_{k+1},\dots,b_n$, such that:
\begin{itemize}
\item The vertices of the arm  $b_1^{\perp},b_2,\dots,b_k$ lie on a circle in the projection plane, satisfying
\begin{itemize}
\item the closing condition if $n-k=$ odd,
\item the diameter condition if $n-k $ = even.
\end{itemize}
\item The arm $b_{k+1},\dots,b_n$ is a zigzag, which projects orthogonally to the diameter of the circle.
\end{itemize}
\end{thm}

\begin{proof}
As in the signed volume case, see Theorem \ref{thm:crit-arms}.
\end{proof}
{\bf Remark 1.} The special case that $p$ is orthogonal to $b_1$ is
included.
In this case we obviously have $b_1^{\perp} = b_1$.\\
If $p$ is parallel to $b_1$ we are in the case of signed volume studied before.\\

{\bf Remark 2.} If we fix only $p$ and not $b_1$ the study of the
signed projected area of the $n$-arm $b_1, \dots, b_n$  is
equivalent to that of the signed volume of the $(n+1)$-arm $p, b_1,
\dots, b_n$.
 We state this:

\begin{thm}\textbf{(General projection on plane)}
The critical points of $PA$  up to "mirror-symmetry" are as follows: \\
There exits a division of the n-arm into two subarms $b_1,\dots,b_k$
and $b_{k+1},\dots,b_n$, such that:
\begin{itemize}
\item The vertices of the arm  $b_1,b_2,\dots,b_k$ lie on a circle in the projection plane, satisfying
\begin{itemize}
\item the closing condition if $n-k=$ odd,
\item the diameter condition if $n-k $ = even.
\end{itemize}
\item The arm $b_{k+1},\dots,b_n$ is a zigzag, which projects orthogonally to the diameter of the circle.\qed
\end{itemize}
\end{thm}

\section{Gram matrices and moduli space}\label{SectionGram}

One way to study the moduli space of $n$-arms in ${\mathbb R}^n$ is
to use the Gram matrix. This has an advantage that there is a direct
relation with the volume.

Given a set  of  vectors, the Gram matrix $ G$ is the matrix of all
possible inner products. Let $B$ be the matrix whose columns are the
arm vectors $b_1,\dots,b_n$. Then the Gram matrix is $G=B^{t}B$. Its
determinant is the square of the volume of the simplex spanned by
these vectors:
   $$\det G = (V)^2.$$

The  Gram matrix is always a positive semi definite symmetric matrix and any positive semi definite symmetric matrix is the  Gram matrix of some $B$ .
If G is positive definite it determines $B$  up to isometry.  

 In our case of $n$-arm in ${\mathbb R}^n$ the inner
products $(b_i.b_i)$ are the fixed numbers $b_i^2$. The other
entries of the Gram matrix we consider as variables $x_{ij}$. Its
determinant is:
$$
\begin{vmatrix}

b_1^2  & x_{12} & x_{13} &   &  & & x_{1n}\\

x_{12}  &b_2^2 & x_{23} &    &   & &x_{2n}\\

x_{13} &  x_{23}& b_3^2 &   &   &    & x_{3n} \\
       &        &       &   &  x_{ij}  &   &      \\
       &        &       & x_{ij}  & & &    \\
       &        &       &    &   &  & &   \\
x_{1n} &        &       &    &   &  &   b_n^2 \\

\end{vmatrix}
$$
For a given  $n$-arm, Gram matrix is contained in a subspace of
dimension $\frac{n(n-1)}{2}$. \\

{\bf Remark.} Note that the equivalence is only up to isometry and
not with respect to orientation. On the set {\small GRAM} of all
Gram matrices we will consider $|V|$. In order to treat the oriented
version we have to take two copies of {\small GRAM} and to glue it
on the common boundary. The set {\small GRAM} is contained in the product of
intervals $-b_ib_j \leq x_{ij} \leq b_ib_j$.

In \cite{merm} diagonals are used as coordinates of the moduli
space. {\small GRAM} is related to that description by the cosine rule:
$$ d_{ij}= b_i^2 + b_j^2 - 2x_{ij}. $$
Note that $G$ is differentiable on the entire space
$\mathbb{R}^{n(n-1)/2}$. In turn, $|V|$ is defined on {\small GRAM},
but is only differentiable on the interior $\{|V|>0\}$. What happens
on the boundary?\\

\noindent
We consider next the 3 dimensional case and use the notations from section \ref{3arm}.
$$
\det G =
\begin{vmatrix}
a^2  & z & y\\
z  & b^2 & x \\
y &  x & c^2 \\
\end{vmatrix}  = 2xyz - a^2x^2 - b^2y^2 - c^2 z^2 + a^2b^2c^2 = 0
$$
In figure \ref{fig:gram} this equation  is visualized. 
Note that {\small GRAM} is equal to the intersection $\{\det G \ge 0\}$ with the box defined by $\{|x|<bc,\; |y| < ac,\; |z|< ab \}$. The boundary of the box intersects $\det G = 0$ only in four points.

The critical points of $\det G$ are given by the conditions\\
$\partial \det G / \partial x = 2 (yz- a^2x) = 0 $ , \\
$\partial \det G / \partial y = 2 (xz - b^2y) = 0 $ ,\\
$\partial \det G / \partial z = 2 ( xy - c^2z) = 0 .$\\

We find the following critical points of $\det G$:
\begin{itemize}
\item $(x,y,z) = (0,0,0)$  : maximum $a^2b^2c^2$ (index 3)
\item $(x,y,z) = (bc,ac,ab), (-bc, ac,-ab), (-bc, -ac, ab)$ or $(bc, -ac, - ab)$
(just the four intersection points mentioned above).

The critical value is equal to $ 0$. What are the types of these 4
critical points?
We compute the Hessian matrix and its determinant:\\
$$
\det H =
\begin{vmatrix}
-a^2  & z & y\\
z  & -b^2 & x \\
y &  x & -c^2 \\
\end{vmatrix}
$$

Note that $\det H(x,y,z) = - \det G(-x,-y,-z)$.\\
Each of our 4 critical points is non-degenerate; since $\det H \neq
0$. The Morse index is 2. Note also that they are related to aligned
situations.
\end{itemize}

\begin{figure} \label{fig:gram}
\centering
\includegraphics[width=10 cm]{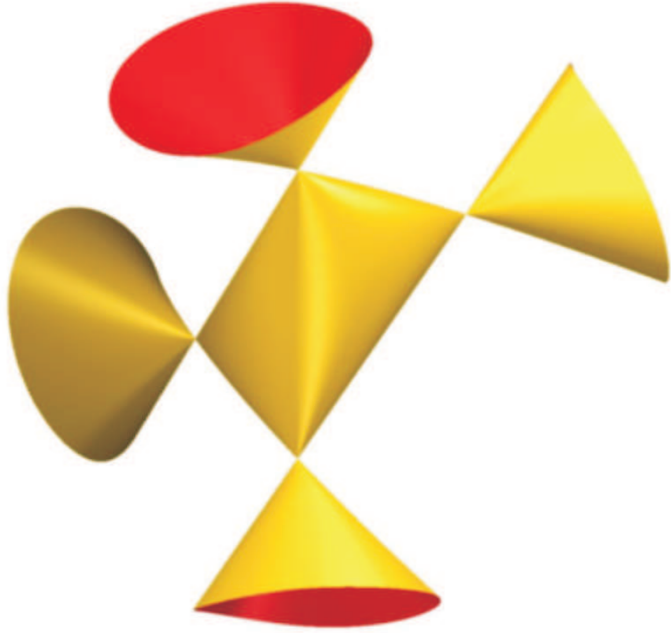}
\caption{Zero locus of the determinant of $G$. The compact region
corresponds to the set of Gram matrices.  (The figure is produced by
SINGULAR software.) } \label{Fig_gram_matrix}
\end{figure}

The local behavior of the level surfaces near the critical level can
be studied with the local formula:
$$ \det G =- \zeta_1^2 - \zeta_2^2 + \zeta_3^2 .$$
Its zero level is a quadratic cone. We restrict ourselves by points
inside the box. Near the singular points we have a homeomorphism:
$$ (\det G)^{-1}[0,\epsilon] = (\det G)^{-1}[\epsilon] \times [0,\epsilon] $$
For the non-critical points this is is guaranteed by the regular
interval theorem; so the product structure is global. We have shown
the following:

\begin{prop} (Fig. \ref{Fig_gram_matrix}) The closure of the component
of $G^{-1}(0,a^2b^2c^2)$,which contains $(0,0,0)$ is a topological
$3$-ball. Its boundary is a topological 2-sphere (differentiable
outside 4 critical points).\qed
\end{prop}
This component is exactly the set {\small GRAM}. Moreover, in this
3-dimensional
 case {\small GRAM} is equivalent (up to isometry) to the set of triples of arm vectors. \\

Since we have $\det G = |V|^2$, the both functions have the same
level curves on the domain of common definition. So the above
proposition tell us that the (unoriented) moduli space of 3-arm is a
topological disc. By gluing two copies of {\small GRAM} along the
common boundary we get:

\begin{thm}
The  oriented moduli space of 3-arms in  ${\mathbb R}^3$ is a
3-sphere. $V$ is an exact topological Morse function on this space
with precisely two Morse critical points.\qed
\end{thm}

\end{document}